\documentclass[12pt, reqno, twoside, letterpaper]{amsart}

\usepackage[bibnixHIDE, arXivON]{sparser-variance}

\title[Sparser variance for primes in arithmetic progression]
      {Sparser variance for primes in arithmetic progression}

\author[R.\ C.\  Baker]{Roger Baker}

\address{Department of Mathematics, 
         Brigham Young University, 
         Provo UT, USA.}

\email{baker@math.byu.edu}

\author[T.\ Freiberg]{Tristan Freiberg}

\address{Department of Pure Mathematics, 
         University of Waterloo, 
         Waterloo ON, CANADA.}

\email{tfreiberg@uwaterloo.ca}

\date{\today}

\begin{document}


\begin{abstract}
We obtain an analog of the Montgomery--Hooley asymptotic formula 
for the variance of the number of primes in arithmetic 
progressions.
In the present paper the moduli are restricted to the sequences of 
integer parts $[F(n)]$, where $F(t) = t^c$ ($c > 1$, 
$c \not\in \NN$) or $F(t) = \exp\big((\log t)^{\gamma}\big)$ 
($1 < \gamma < 3/2$).
\end{abstract}



\maketitle


\section{Introduction}
 \label{sec:1}
 
Let $F$ be a real differentiable function on $(1,\infty)$ with the 
property that 
\[
 F(y) \ge 2, 
  \quad 
   F'(y) \ge 1
    \quad 
     (y \ge y_0(F)).
\]
We write 
\[
 f(y) = [F(y)].
\]
We are concerned with the remainders 
\[
 E(x; h, \ell)
  = 
   \sums[p \le x][p \equiv \ell \bmod h] \log p
    -
     \frac{x}{\phi(h)}, 
      \quad 
       (\ell, h) = 1, 
\]
where $x$ is large and the moduli $h$ are restricted to the values 
$f(k)$.
Here and below, $p$ denotes a prime number.
Let $V_F(x,y)$ denote the variance
\[
 V_F(x,y)
  = 
   \sum_{y_0(F) < k \le y} 
    F'(k)
     \sums[\ell = 1][(\ell,f(k)) = 1]^{f(k)}
      E(x; f(k), \ell)^2.
\]
When $F(k) = f(k) = k$, the Montgomery--Hooley theorem 
\cite{HOO75, MON70} states that for $1 \le y \le x$, 
\[
 V_F(x,y) 
  = 
   xy\log y 
    + 
     c_0xy 
      + 
       O\big(x^{1/2}y^{3/2} + x^2(\log x)^{-A}\big).
\]
Here and below, $A$ denotes an arbitrary positive constant; we 
take $A > 1$.
(Implied constants depend on $A$ throughout: dependencies on 
constants such as $c$ are indicated in context.)
The constant $c_0$ can be given explicitly.
This asymptotic formula was generalized by Br\"udern and Wooley 
\cite{BW11} to the case where $F = f$ is an integer-valued 
polynomial of degree $\ge 2$ with positive leading coefficient. 
They found that for $1 \le F(y) \le x$, 
\begin{equation}
 \label{eq:1.1}
  V_F(x,y) 
   = 
    xf(y)\log f(y)
     + 
      C(f)xf(y)
       +
        O\big(x^{1/2}f(y)^{3/2} + x^2(\log x)^{-A}\big).             
\end{equation}

In the present paper, we give two further variants of the 
Montgomery--Hooley theorem.

\begin{theorem}
 \label{thm:1}
Let $F(k) = k^c$, where $c > 1$ and $c \not\in \NN$.
Then \eqref{eq:1.1} holds for $F(y) \le x$, with $C(f)$ replaced 
by a constant $C$ independent of $f$.
\end{theorem}
\noindent
The constant $C$ is evaluated in Section \ref{sec:5} (see 
\eqref{eq:5.10}).

\begin{theorem}
 \label{thm:2}
Let $F(k) = \exp\big((\log k)^{\gamma}\big)$, where 
$1 < \gamma < 3/2$.
Let $C_1 > 1/(3 - 2\gamma)$.
For $F(y) \le x$ we have, with $C$ as in Theorem \ref{thm:1},
\begin{align*}
 & 
  V_F(x,y) 
   -
    V_F(x,\exp((\log \log x)^{C_1}))
 \\
 & 
  \hspace{30pt}
 =
   xf(y)\log f(y)
     + 
      Cxf(y)
       +
        O\big(x^{1/2}f(y)^{3/2} + x^2(\log x)^{-A}\big).
\end{align*}
\end{theorem}
\noindent%
If we knew more about either Siegel zeros or exponential sums, we 
would not have to omit small moduli in Theorem \ref{thm:2}; see 
\cite[Section 6]{BAK14}.

\vspace{1em}
\noindent%
{\bfseries Acknowledgements.}
The arguments in Sections \ref{sec:4} and \ref{sec:5} are adapted 
from \cite{BW11} with some notable differences.
We thank Trevor Wooley for insightful comments about these 
differences. 
Thanks are also due from R.\ B.\ to the Department of Pure 
Mathematics, University of Waterloo for hospitality, and to the 
Simons Foundation for a Collaboration Grant.


\section*{Notation}

As is customary, $\phi$ denotes Euler's totient function, $\mu$ 
denotes the M\"obius function, $e(\theta)$ abbreviates 
$\e^{2\pi i \theta}$, and $\|t\| = \min(\{t\}, 1 - \{t\})$, where 
$\{t\} = t - [t]$ denotes the fractional part of $t$. 
Throughout, we regard the quantities $c$, $\gamma$ and $A$ as 
fixed and independent of all other quantities: we only assume 
that $c > 1$, $c \not\in \NN$, $1 < \gamma < 3/2$, and $A > 1$ 
(arbitrarily large).
We regard $B$ as fixed, but sufficiently large in terms of $A$ and 
$c$ (respectively, $A$ and $\gamma$) in the case $F(t) = t^c$ 
(respectively, $F(t) = \exp\big((\log t)^{\gamma}\big)$). 
We write $C_1,C_2,\ldots$ for `large' positive constants and 
$c_1,c_2,\ldots$ for `small' positive constants: each $C_i$ 
may depend on $c$, $\gamma$, $A$, $B$, $C_1,\ldots,C_{i - 1}$ and 
$c_1,\ldots,c_{i - 1}$ (indicated in context); likewise for each 
$c_i$.
We view $x$ as a real parameter tending to infinity, and write 
$U \ll V$, $V \gg U$, or $U = O(V)$ to denote that 
$|U| \le \kappa|V|$ for all sufficiently large $x$, where $\kappa$ 
is a constant which may depend on $A$, as well as other fixed 
quantities (indicated in context).
We write `$U \asymp V$' for `$U \ll V$ and $V \ll U$'.


\section{Some lemmas}
 \label{sec:2}
 
Most of these preliminary results come from \cite{BAK14}.
Whether $F$ be as in Theorem \ref{thm:1} or Theorem \ref{thm:2}, 
let us write
\begin{equation}
 \label{eq:2.1}
  S_Q 
   =
    \{[F(n)] : Q < [F(n)] \le 2Q\}.
\end{equation}

\begin{lemma}
 \label{lem:1}
Let $F$ be as in Theorem \ref{thm:1}.
For $Q \in \NN$ and $S_Q$ as in \eqref{eq:2.1}, we have 
\begin{equation}
 \label{eq:2.2}
  \sum_{q \, \in \, S_Q}
   \sums[\ell = 1][(\ell,q) = 1]^{q} 
    E(x;q,\ell)^2
     \ll 
      \frac{|S_Q|x^2}{Q\psL^{2A}}
\end{equation}
provided that $Q \le x\psL^{-B}$.
Here $B$ is a positive constant depending on $A$ and $c$.
\end{lemma}

\begin{proof}
\cite[Theorem 1.1]{BAK14}.
\end{proof}

\begin{lemma}
 \label{lem:2}
Let $F$ be as in Theorem \ref{thm:2}.
For $Q \in \NN$ and $S_Q$ as in \eqref{eq:2.1}, we have 
\eqref{eq:2.2} provided that 
\[
 \exp\big((\log \sL)^{C_2}\big)
  \le 
   Q
    \le 
     x\psL^{-B},
\]
where $C_2 > \gamma/(3 - 2\gamma)$.
Here $B$ is a positive constant depending on $A$ and $\gamma$.
The implied constant \textup{(}in \eqref{eq:2.2}\textup{)} depends 
on $\gamma$, $C_2$ and $A$.%
\end{lemma}

\begin{proof}
\cite[Theorem 1.2]{BAK14}.
\end{proof}

\begin{lemma}
 \label{lem:3}
Let $2 \le N \le N_1 \le 2N$. 
\\
(i)
Let 
$
 F(y) 
  =
   \exp\big((\log y)^{\gamma}\big)
$.
%
Let $0 < \beta < F(N)^{c_2}$ with $0 < c_2 < \gamma - 1$, and 
\begin{equation}
 \label{eq:2.3}
 \beta F'(N) \ge \frac{1}{2}.
\end{equation}
Then 
\[
 \bigg|
  \sum_{N < n \le N_1}
   e\big(\beta F(n)\big)
 \bigg|
  \le 
   C_3 
    N
     \exp\big(-c_3(\log N)^{3 - 2\gamma}\big)
\]
where $C_3,c_3$ depend on $\gamma$ and $c_2$.
\\
(ii)
Let 
$
 F(y) = y^c
$.
%
Let $0 < \beta < N$ and suppose that \eqref{eq:2.3} holds.
Then 
\[
 \bigg|
  \sum_{N < n \le N_1}
   e\big(\beta F(n)\big)
 \bigg|
  \le 
   C_4 
    N^{1 - c_4}
\]
where $C_4,c_4$ depend on $c$ and $c_4 < 1/20$.
\end{lemma}

\begin{proof}
\cite[Lemma 2.2]{BAK14}.
\end{proof}

For the remainder of the paper, $B$ and $c_4$ are as in 
Lemmas \ref{lem:1}--\ref{lem:3} and $C_5,C_6,C_7$ are constants 
satisfying 
\begin{equation}
 \label{eq:2.4}
 C_5 \ge 8A + 1 + B, 
  \quad 
   C_6 \ge C_5 + 12A + 2, 
    \quad 
     C_7 \ge \frac{8A + 2}{c_4}.
\end{equation}

\begin{lemma}
 \label{lem:4}
Let $P = \psL^{C_5}$, $R = x\psL^{-C_6}$.
Let $\alpha > 0$ and suppose there is no rational number $a/q$, 
$(a,q) = 1$, $1 \le q \le P$ satisfying 
\[
 \bigg| \alpha - \frac{a}{q} \bigg|
  \le 
   \frac{1}{qR}.
\]
Suppose that either 

\vspace{0.5em}

\noindent%
(a)
$F(y) = y^c$ 
\textup{(}$c > 1$, $c \not\in \NN$\textup{)}, or    

\vspace{0.5em}

\noindent%
(b) 
$F(y) = \exp\big((\log y)^{\gamma}\big)$ 
\textup{(}$1 < \gamma < 3/2$\textup{)}.

\vspace{0.5em}

\noindent%
Let $\psL^{C_7} \le K \le K_1 \le 2K$ and $M \ge 1/2$, 
$MF(K) \le 2x$.
Suppose further in case (b) that 
\begin{equation}
 \label{eq:2.5}
 \log K > (\log \sL)^{C_1}, 
  \quad 
   C_1 > 1/(3 - 2\gamma).
\end{equation}
Then 
\[
   \sum_{M < m \le 2M}
    \sums[K < k \le K_1][mf(k) \le x]
     e\big(\alpha mf(k)\big)
      \ll
       \frac{Kx}{F(K)\psL^{4A}}.
\]
The implied constant depends on $c,A$ in case (a) and 
$\gamma, C_1, A$ in case (b).
\end{lemma}

\begin{proof}
This follows at once from \cite[Theorem 2.5]{BAK14}.
\end{proof}

\begin{lemma}
 \label{lem:5}
Make the hypotheses of Lemma \ref{lem:4} and suppose further that 
\[
 x\psL^{-B} \le F(K) \le x.
\]
Let $N(K,K_1,q,\ell)$ be the number of solutions to 
\[
 f(k) \equiv \ell \bmod q, 
  \quad 
   K < k \le K_1.
\]
Then for $1 < q \le \psL^{C_5}$, we have 
\[
 N(K,K_1,q,\ell) 
  =
   \frac{K_1 - K}{q} + O\big(K\psL^{-4A}\big).
\]
The implied constant depends on $c, A, B$ in case (a) and 
$\gamma, C_1, A, B$ in case (b).

\end{lemma}

\begin{proof}
We have 
\[
 N(K,K_1,q,\ell)
  =
   \sum_{K < k \le K_1}
    \frac{1}{q}
     \sum_{a = 1}^q 
      e\bigg(\frac{a(f(k) - \ell)}{q}\bigg).
\]
Separating the contribution from $a = q$, 
\begin{equation}
 \label{eq:2.6}
 N(K,K_1,q,\ell)
  -
   \frac{K_1 - K}{q}
    =
     \frac{1}{q}
      \sum_{a = 1}^{q - 1}
       e\bigg(-\frac{a\ell}{q}\bigg)
        \sum_{K < k \le K_1}        
         e\bigg(\frac{af(k)}{q}\bigg).
\end{equation}

The remainder of the proof is a variant of the proof of 
\cite[Theorem 2.5]{BAK14} in the case $M = 1/2$; we have, for 
$1 \le a < q$,
\begin{equation}
 \label{eq:2.7}
 \sum_{K < k \le K_1}
  e\bigg(\frac{af(k)}{q}\bigg)
   =
    T_1(\alpha) + O\big(T_2(\alpha)\big),
\end{equation}
where, with $H = \psL^{4A + 1}$,
\begin{align*}
 T_1(\alpha)
  & 
   =
    \sums[h \, \in \, \ZZ][|h + \frac{a}{q}| \le H]    
     c_h\bigg(\frac{a}{q}\bigg)
      \sum_{K < k \le K_1}
       e\bigg(\bigg(h + \frac{a}{q}\bigg)F(k)\bigg)
 \\
 T_2(\alpha)
  & 
   =
    \sum_{K < k \le K_1}
     \min\bigg(\sL, \frac{1}{H\|F(k)\|}\bigg).
\end{align*}
Here, 
\[
 c_h(\beta)
  \defeq 
   \frac{1 - e(-\beta)}{2\pi i (h + \beta)}.
\]
Just as in the proof of \cite[Lemma 2.4]{BAK14} we have 
\begin{align}
 \label{eq:2.8}
 T_2(\alpha)
  & 
  \ll
   K\psL^{-4A},
 \\
  \label{eq:2.9}
 T_1(\alpha)
  & 
   \ll
    \sums[h \in \ZZ][|h + \frac{a}{q}| \le H]
     \bigg|
      \sum_{K < k \le K_1}
       e\bigg(\bigg(h + \frac{a}{q}\bigg)F(k)\bigg) 
     \bigg|.
\end{align}
Note that $|h + \frac{a}{q}| < F(K)^{c_2}$, while 
\begin{equation}
 \label{eq:2.10}
 \bigg|h + \frac{a}{q}\bigg| 
  F'(2K)
   \ge 
    \frac{1}{2}.
\end{equation}
To see this, 
\begin{equation}
 \label{eq:2.11}
  \bigg|h + \frac{a}{q}\bigg|  
   \ge 
    \psL^{-C_5}, 
\end{equation}
while
\begin{equation}
 \label{eq:2.12}
 F'(2K) 
  \gg 
   \frac{F(K)}{K}
    \gg
     \begin{cases}
      F(K)^{1 - \frac{1}{c}} & \text{in case (a)} \\
      F(K)^{1/2}             & \text{in case (b),}
     \end{cases}
\end{equation}
and 
\begin{equation}
 \label{eq:2.13}
 F(K) 
  \gg
   x\psL^{-B}.
\end{equation}
Combining \eqref{eq:2.11}--\eqref{eq:2.13} yields \eqref{eq:2.10}.
We now use Lemma \ref{lem:3}, noting that, in case (b), 
\[
 \exp\big(-c_3(\log K)^{3 - 2\gamma}\big) 
  \ll 
   \psL^{-9A}
\]
since $(\log K)^{\gamma} > \frac{1}{2}\log x$.
This gives 
\[
 \sum_{K < k \le K_1}
  e\bigg(\bigg(h + \frac{a}{q}\bigg)F(k)\bigg)
   \ll 
    \psL^{-9A}
\]
and the lemma now follows from \eqref{eq:2.7}--\eqref{eq:2.9}.
\end{proof}

\begin{lemma}
 \label{lem:6}
Let 
\[
 c_q(h) 
  =
   \sums[a = 1][(a,q) = 1]^q
    e\bigg(\frac{ah}{q}\bigg)  
     \quad 
      \text{and}
       \quad 
        w_h(q) 
         =
          \frac{1}{q}
           \sum_{a = 1}^q c_q(ha).
\]
Then for squarefree $q$, 
\[
 w_h(q) 
  =
   \begin{cases}
    \phi(q) & \text{when $q \mid h$} \\
    0       & \text{when $q \nmid h$.}
   \end{cases}
\]
\end{lemma}

\begin{proof}
This is a special case of \cite[Lemma 4.2]{BW11}.
\end{proof}


\section{First stage of proof of Theorems \ref{thm:1} and \ref{thm:2}}
 \label{sec:3}

This section is similar to material in \cite{BAK14, BW11, MP05}, 
but there are enough differences to give the details.
Define $y_1$ by 
\[
 F(y_1) = x\psL^{-B}.
\]
We are concerned with values of $k$ satisfying 
\begin{equation}
 \label{eq:3.1}
 y_1 < k \le y, 
  \quad 
   \text{where}
    \quad 
     F(y) \le x.
\end{equation}
We note that 
\begin{equation}
 \label{eq:3.2}
 \frac{f(k)}{k}
  \le 
   \frac{F(k)}{k}
    \ll
     F'(k)
      \ll
       \frac{F(k)}{k}\sL
        \ll 
         \frac{f(k)}{k}\sL.
\end{equation}
Our objective is to evaluate 
\begin{equation}
 \label{eq:3.3}
 V_F'(x,y)
  \defeq 
   \sum_{y_1 < k \le y} 
    F'(k)
     \sums[\ell = 1][(\ell,f(k)) = 1]^{f(k)}
      E(x; f(k),\ell)^2
\end{equation}
asymptotically.
Let 
\[
 \theta(x;k,\ell)
  =
   \sums[p \le x][p \equiv \ell \bmod k] \log p,
\]
\begin{equation}
 \label{eq:3.4}
 \Phi_F(z,y)
  =
   \sum_{z < k \le y} 
    \frac{F'(k)}{\phi(f(k))}.
\end{equation}
Opening the square in \eqref{eq:3.3}, we find that 
\begin{equation}
 \label{eq:3.5}
 V_F'(x,y) 
  =
   S_1 - 2xS_2 + x^2\Phi_F(y_1,y),
\end{equation}
where 
\begin{align*}
 S_1
  & 
   =
    \sum_{y_1 < k \le y} 
     F'(k) 
      \sums[\ell = 1][(\ell,f(k)) = 1]^{f(k)}
       \theta(x;f(k),\ell)^2,
 \\
 S_2 
  & 
   =
    \sum_{y_1 < k \le y}
     \frac{F'(k)}{\phi(f(k))}
      \sums[\ell = 1][(\ell,f(k)) = 1]^{f(k)}
       \theta(x; f(k),\ell).
\end{align*}
Using the prime number theorem and the fact that $\ll \sL$ primes 
divide $f(k)$ ($k \le y$), we rewrite $S_2$ as 
\begin{align*}
 S_2
  & 
   =
    \sum_{y_1 < k \le y}
     \frac{F'(k)}{\phi(f(k))}
      \bigg(
       x + O\big(x\psL^{-3A}\big)
      \bigg)
 \\
  & 
   =
    x\Phi_F(y_1,y)
     + 
      O
      \bigg(
       x
        \psL^{-3A + 1}
         \sum_{y_1 < k \le y}
          \frac{f(k)}{k\phi(f(k))}
      \bigg)
\end{align*}
in view of \eqref{eq:3.2}.
Since $\phi(f(k)) \gg f(k)\psL^{-1}$, we find that 
\begin{equation}
 \label{eq:3.6}
 S_2 
  =
   x\Phi_F(y_1,y)
    +
     O\big(x\psL^{-3A + 3}\big).
\end{equation}

We may easily derive the relation 
\begin{equation}
 \label{eq:3.7}
  \sums[\ell = 1][(\ell,f(k)) = 1]^{f(k)}
  \theta(x;f(k),\ell)^2
   =
   \underset
    {
     p_1 \equiv p_2 \bmod f(k)
    }
    {
    \sums[p_1 \le x][p_1 \nmid f(k)]
     \sums[p_2 \le x][p_2 \nmid f(k)]
    }    
      (\log p_1) 
       (\log p_2).     
\end{equation}
The conditions $p_1 \mid f(k)$, $p_1 \equiv p_2 \bmod f(k)$ imply 
that $p_1 = p_2$.
We may accordingly ignore the constraint $p_j \nmid f(k)$ 
($j = 1,2$) when considering the off-diagonal terms. 
Consequently, 
\begin{equation}
 \label{eq:3.8}
  \sums[\ell = 1][(\ell,f(k)) = 1]^{f(k)}
  \theta(x; f(k), \ell)^2
   =
    \sums[p \le x][p \nmid f(k)] (\log p)^2 
     \hspace{3pt}
      + 
       \hspace{5pt} 
       2
        \hspace{-8pt}
         \sums[p_1 < p_2 \le x][p_1 \equiv p_2 \bmod f(k)] 
         (\log p_1)(\log p_2).
\end{equation}

We note the bounds 
\begin{equation}
 \label{eq:3.9}
 \sum_{k = k_0}^{k_1} F'(k - 1) 
  \le 
   \sum_{k = k_0}^{k_1} \big(F(k) - F(k - 1)\big)
    \le 
     \sum_{k = k_0}^{k_1} F'(k)     
\end{equation}
valid for any $k_0,k_1 \in \NN$, $k_0 \le k_1$.
It follows that
\[
 \sum_{y_1 < k \le y} F'(k)
  =
   F(y) + O\big(x\psL^{-2A}\big),
\]
\begin{equation}
 \label{eq:3.10}
 \sum_{y_1 < k \le y} F'(k) 
  \sum_{p \le x} (\log p)^2 
   =
    F(y) 
     \sum_{p \le x} (\log p)^2 
      +
       O\big(x^2\psL^{-A}\big).
\end{equation}
We let 
\[
 S_0 
  \defeq 
   \sum_{y_1 < k \le y} 
    F'(k) 
     \sums[p_1 < p_2 \le x][p_1 \equiv p_2 \bmod f(k)] 
      (\log p_1)(\log p_2).
\]
We deduce easily from \eqref{eq:3.8}, \eqref{eq:3.10} that 
\[
 S_1 
  = 
   2S_0 
    + 
     f(y) 
      \sum_{p \le x} (\log p)^2 
       +
        O\big(x^2\psL^{-A}\big).
\]
Combining this with \eqref{eq:3.5}, \eqref{eq:3.6}, we have 
\begin{equation}
 \label{eq:3.11}
 V_F'(x,y)
  = 
   2S_0 
    -
     x^2
      \Phi_F(y_1,y) 
       + 
        f(y)
         \sum_{p \le x} (\log p)^2
          +
           O\big(x^2\psL^{-A}\big).
\end{equation}

Let 
\[
 T(\alpha)
   =
    \sum_{y_1 < k \le y} F'(k) 
     \sum_{h \le x/f(k)} e\big(\alpha hf(k)\big),
 \quad 
 U(\alpha)
   =
    \sum_{p \le x} 
     (\log p)e(\alpha p).
\]
It is straightforward to verify that 
\begin{equation}
 \label{eq:3.12}
 S_0 
  =
   \int_0^1 T(\alpha) |U(\alpha)|^2 \dd \alpha.
\end{equation}
Let $P,R$ be as in Lemma \ref{lem:4}.
Define the major arcs $\gM$ to be the union of the pairwise 
disjoint intervals 
\[
 \bigg\{
         \alpha : 
          \bigg|\alpha - \frac{a}{q} \bigg| 
           \le \frac{1}{qR} 
 \bigg\} 
  \quad 
   (1 \le a \le q \le P, \, (a,q) = 1)
\]
and the minor arcs $\gm$ by 
\[
 \gm 
  =
   \bigg[\frac{1}{R}, 1 + \frac{1}{R} \bigg] \, \setminus \, \gM.
\]
A splitting-up argument gives 
\begin{equation}
 \label{eq:3.13}
 J_{\gm}
  \defeq 
   \int_{\gm} T(\alpha) |U(\alpha)|^2 \dd \alpha
    \ll 
     \psL^2 
      \int_{\gm} 
       |U(\alpha)| T^*(\alpha) \dd \alpha 
\end{equation}
where $T^*(\alpha)$ is the contribution to $T(\alpha)$ from 
$K < k \le K_1$ and $M < h \le 2M$.
Here $1/2 \le M \le x/f(K)$ while $y_1 \le F(K) < y$ and 
$K < K_1 \le 2K$.
Moreover, 
\begin{equation}
 \label{eq:3.14}
 \int_{\gm} |U(\alpha)|^2 \dd \alpha 
  \le 
   \int_{\frac{1}{R}}^{1 + \frac{1}{R}}
    |U(\alpha)|^2 \dd \alpha 
     =
      \sum_{p \le x}(\log p)^2 
       \ll
        x\sL.
\end{equation}

Now, 
\[
 T^*(\alpha)
  = 
   \int_K^{K_1} 
    F'(t) \dd S(t) 
\]
with 
\[
 S(t) 
  =
   \sum_{y_1 < k \le t} 
    \sums[h \le x/f(k)][M < h \le 2M] 
     e\big(\alpha h f(k)\big).
\]
Since $F''$ is monotonic, taking sup norms on $[K,K_1]$ we have 
\begin{align*}
 T^*(\alpha) 
 & 
 =
   \big[F'(t) S(t)\big]_K^{K_1} 
    -
     \int_{y_1}^y F''(t) S(t) \dd t 
 \\
  & 
      \ll 
       \|F'\|_{\infty}
        \|S\|_{\infty} 
         \ll
          \frac{K}{F(K)\psL^{4A}}
           \|F'\|_{\infty}
            \ll
             x\psL^{-3A},
\end{align*}
where we have used Lemma \ref{lem:4} and \eqref{eq:3.2} for the 
second last and last bounds respectively.
Combining this with \eqref{eq:3.13}, \eqref{eq:3.14} we have 
\begin{equation}
 \label{eq:3.15}
 J_{\gm}
  \ll
   x^2\psL^{-A}.
\end{equation}

We turn to the major arcs, beginning with 
\[
 J_{\gM} 
  \defeq 
   \int_{\gM} T(\alpha) |U(\alpha)|^2 \dd \alpha 
    =
     \sum_{q \le P}
      \sums[b = 1][(b,q) = 1]^q
       \int_{-\frac{1}{qR}}^{\frac{1}{qR}}
        \bigg|U\bigg(\frac{b}{q} + \beta\bigg)\bigg|^2
         T\bigg(\frac{b}{q} + \beta\bigg) 
          \dd \beta.
\]
Let 
\[
 v(\beta) 
  =
   \sum_{m \le x} e(m\beta).
\]
From Vaughan \cite[Lemma 3.1]{VAU97} we see that in the last 
integral, 
\[
 \bigg|U\bigg(\frac{b}{q} + \beta\bigg)\bigg|^2
  =
   \frac{\mu(q)^2}{\phi(q)}|v(\beta)|^2
    +
     O\big(x^2\exp\big(-c_5\psL^{1/2}\big)\big).
\]
Now 
\[
 T(\alpha) 
  \ll
   x
    \sum_{y_1 < k \le y} 
     \frac{F'(k)}{f(k)}
      \ll
       x
        \psL
         \sum_{y_1 < k \le y} 
          \frac{1}{k} 
           \ll 
            x\psL^2.
\]
Hence
\begin{align*}
 & 
 \bigg|U\bigg(\frac{b}{q} + \beta\bigg)\bigg|^2
  T\bigg(\frac{b}{q} + \beta\bigg) 
 \\
 & 
  \hspace{30pt}
  =
    \frac{\mu(q)^2}{\phi(q)}|v(\beta)|^2
     T\bigg(\frac{b}{q} + \beta\bigg) 
      +
       O\big(x^3 \psL^2\exp\big(-c_5\psL^{1/2}\big)\big),
\end{align*}
\begin{align*}
 J_{\gM}
 & 
  =
   \sum_{q \le P}
    \frac{\mu(q)^2}{\phi(q)^2}
     \sums[b = 1][(b,q) = 1]^q
      \int_{-\frac{1}{qR}}^{\frac{1}{qR}} 
       |v(\beta)|^2T\bigg(\frac{b}{q} + \beta\bigg) \dd \beta 
  \\
  & 
   \hspace{90pt}
    +
     O\big(x^2 \psL^{C_5 + C_6 + 2}\exp\big(-c_5\psL^{1/2}\big)\big).
\end{align*}
Since $v(\beta) \ll |\beta|^{-1}$ on $(-1/2,1/2)$ we extend the 
integral to this interval, introducing an error%
\begin{align*}
 & 
 \ll 
  \sum_{q \le P}
   \frac{\mu(q)}{\phi(q)}
    x 
     \psL^2
      \int_{\frac{1}{qR}}^{\frac{1}{2}} 
       \frac{1}{\beta^2}
        \dd \beta
         \ll
          x
           \psL^2 
            \sum_{q \le P} 
             \frac{\mu(q)^2}{\phi(q)}
              qR 
               \ll 
                x^2
                 \psL^{C_5 - C_6 + 3}
 \\
 & 
                  \ll
                   x^2
                    \psL^{-A}.
\end{align*}
This yields 
\begin{equation}
 \label{eq:3.16}
 J_{\gM}
  =
   \sum_{q \le P} 
    \frac{\mu(q)^2}{\phi(q)^2} 
     H(q)
      +
       O\big(x^2\psL^{-A}\big)
\end{equation}
where 
\[
 H(q)
  =
   \sums[b = 1][(b,q) = 1]^q
    \int_{-\frac{1}{2}}^{\frac{1}{2}} 
     |v(\beta)|^2 
      T\bigg(\frac{b}{q} + \beta\bigg) 
       \dd \beta.
\]

By orthogonality, 
\begin{align}
 \label{eq:3.17}
  \begin{split}
 H(q)
  & 
   =
    \sums[b = 1][(b,q) = 1]^q
     \sum_{y_1 < k \le y} 
      F'(k) 
       \sum_{h \le x/f(k)} 
        \underset%
         {
          n_1 - n_2 = hf(k)
         }
         {
          \sum_{n_1 \le x} 
           \sum_{n_2 \le x}
         }  
          e\bigg(\frac{bhf(k)}{q}\bigg)
 \\
 & 
  =
   \sum_{y_1 < k \le y}
    F'(k) 
     \sum_{h \le x/f(k)} 
      c_q\big(hf(k)\big)
       \big([x] - hf(k)\big).
  \end{split}
\end{align}
Replacing $[x]$ by $x$ introduces an error in \eqref{eq:3.17} of 
\[
 \ll
  \psL^{C_5} 
   \sum_{y_1 < k \le y} 
    F'(k) 
     \frac{x}{f(k)} 
      \ll
       x\psL^{C_5 + 2}
\]
by \eqref{eq:3.2}.
Combining this with \eqref{eq:3.12}, \eqref{eq:3.15} and 
\eqref{eq:3.16} we reach the expression 
\begin{equation}
 \label{eq:3.18}
  S_0 
   =
    M_0 
     + 
      O\big(x^2\psL^{-A}\big),
\end{equation}
where 
\[
 M_0
  =
   \sum_{q \le P}
    \frac{\mu(q)^2}{\phi(q)^2}
     \sum_{y_1 < k \le y}
      F'(k)
       \sum_{h \le x/f(k)} 
        c_q\big(h f(k)\big)
         (x - h f(k)).
\]
 

\section{Proof of Theorems \ref{thm:1} and \ref{thm:2}: second stage}
 \label{sec:4}
 
We first show that $M_0$ can be simplified to the form 
\begin{equation} 
 \label{eq:4.1} 
 M_0 
  =
   \sum_{h \le x/f(y_1)} 
    \frac{h}{\phi(h)} 
     \int_{f(y_1)}^{f(y(h))} 
      (x - ht) 
       \dd t 
        +
         O\big(x^2\psL^{-A}\big)       
\end{equation}
where $y(h)$ is defined by 
\[
 F(y(h)) 
  = 
   \min\bigg(F(y), \frac{x}{h}\bigg).
\]

Sorting the integers $k$ according to the value of $f(k) \bmod q$, 
\begin{equation}
 \label{eq:4.2}
  M_0 
   =
    \sum_{q \le P}
     \frac{\mu(q)^2}{\phi(q)^2}
      \sum_{h \le x/f(y_1)} 
       \sum_{b = 1}^q 
        c_q(hb)S_0\big(y(h),b\big),
\end{equation}
where 
\begin{equation}
 \label{eq:4.3} 
  S_0(z,b) 
   =
    \sums[y_1 < k \le z][f(k) \equiv b \bmod q] 
     F'(k)(x - hf(k)).
\end{equation}

Let 
\begin{equation}
 \label{eq:4.4} 
  S_1(z,b)
   =
    \sums[y_1 < k \le z][f(k) \equiv b \bmod q]
     F'(k)(x - hF(k)),
\end{equation}
then 
\begin{equation}
 \label{eq:4.5} 
  S_0(z,b) 
   -
    S_1(z,b)
   \ll 
    h
     \sum_{y_1 < k \le z}
      F'(k)
       \ll
        h
         F\big(y(h)\big)
          \ll 
           x
\end{equation}
for $y_1 < z \le y(h)$, by \eqref{eq:3.9}.
Now let 
\[
 S_2(z)
  =
   \sum_{y_1 < k \le z} 
    F'(k)(x - hF(k))
\]
and, in the notation of Lemma \ref{lem:5}, 
\[
 N_b(t) 
  =
   N(y_1,t,q,b)
    -
     \frac{1}{q} 
      \sum_{y_1 < n \le t} 1.
\]
We have, for $z \in (y_1,y(h)]$, 
\begin{align*}
 S_1(z,b) 
  -
   \frac{S_2(z)}{q}
    & 
     =
      \int_{y_1}^z 
       F'(t)(x - hF(t)) \dd N_b(t) 
 \\ 
  & 
   =
    \big[F'(t)(x - hF(t))N_b(t)\big]_{y_1}^{z}
 \\
 & 
  \hspace{60pt}
    -
      x\int_{y_1}^z N_b(t) F''(t) \dd t
       +
        h\int_{y_1}^z N_b(t)
         \frac{\dd}{\dd t}(F(t)F'(t))\dd t
 \\
  & 
   =
    T_1 - T_2 + T_3,
\end{align*}
say.

Now, 
\[
 N_b(t)
  \ll
   y(h) 
    \psL^{-3A}
\]
by Lemma \ref{lem:5} and a splitting-up argument.
Hence 
\begin{align}
 \label{eq:4.6} 
  T_1 
   & 
    \ll 
     x F'(y(h)) y(h)\psL^{-3A} \ll \frac{x^2}{h}\psL^{-2A}, 
  \\
 \label{eq:4.7} 
  T_2 
   & 
    \ll 
     x F'(y(h)) y(h)\psL^{-3A} \ll \frac{x^2}{h}\psL^{-2A},
  \\
 \label{eq:4.8} 
  \begin{split}
  T_3 
   & 
    \ll  
     h y(h) \psL^{-3A} F(y(h)) F'(y(h)) 
  \\
   & 
      \ll 
       x y(h) \psL^{-3A} F'(y(h)) 
        \ll
         \frac{x^2}{h} \psL^{-2A}.  
  \end{split}
\end{align}

Next we must estimate the difference 
\[
 D(z)
  \defeq 
   S_2(z) 
    -
     \int_{y_1}^z 
      F'(u) 
       (x - hF(u)) \dd u
\]
for $z \in (y_1,y(h)]$.
By Euler's formula, 
\begin{align*}
 & 
 |D(z)|
 \\
 & 
  \le 
   \int_{y_1}^z 
    \bigg| 
     \frac{\dd}{\dd u} \big(F'(u)(x - hF(u))\big)
    \bigg| 
     \dd u
      +
       \big|F'(y_1)(x - hF(y_1))\big|
        +
         \big|F'(z)(x - hF(z))\big|
 \\
 & 
  =
    U_1 + U_2 + U_3,
\end{align*}
say.
We have 
\begin{align}
 \label{eq:4.9}
  \begin{split}
   U_1 
    & 
     \ll 
      x
       \bigg| 
        \int_{y_1}^z 
         \frac{\dd}{\dd u}
          \big(F'(u)\big)
           \dd u
       \bigg|
        +
         h
          \bigg|
           \int_{y_1}^z 
            \frac{\dd}{\dd u} 
             \big(F'(u) F(u)\big) \dd u
          \bigg|
   \\
    & 
     \ll 
      x 
       F'(y(h)) 
        + 
         h
          F'(y(h)) 
           F(y(h))
   \\
    & 
     \ll 
      \frac{x F(y(h))\sL}{y_1}
   \\
   & 
       \ll 
        \frac{x^2\psL^{-2A}}{h},
  \end{split}
\end{align}
\begin{equation}
 \label{eq:4.10} 
  U_2 + U_3
   \ll
    xF'(y(h)) 
     \ll 
      \frac{x^2\psL^{-2A}}{h}.
\end{equation}

Assembling \eqref{eq:4.2}--\eqref{eq:4.10}, we obtain 
\begin{align}
 \label{eq:4.11}
  \begin{split}
   M_0
  & 
   =
     \sum_{q \le P}
      \frac{\mu(q)^2}{\phi(q)^2} 
       \sum_{h \le x/f(y_1)} 
        \sum_{b = 1}^q 
         c_q(hb) 
          \int_{y_1}^{y(h)} 
           F'(u)(x - hF(u)) \dd u
  \\
   & 
    \hspace{150pt}
     +
      O
       \bigg(
        \frac{x^2}{\psL^{2A}} 
         \sum_{q \le P}
          \frac{\mu(q)^2}{\phi(q)}
           \sum_{h \le 2\psL^B} 
            \frac{1}{h}
       \bigg).
  \end{split}
\end{align}
The error here is $O\big(x^2 \psL^{-A}\big)$.
Using a substitution in the integral, and applying 
Lemma \ref{lem:6}, the main term in \eqref{eq:4.11} is 
\[
 \sum_{q \le P}
  \frac{\mu(q)^2}{\phi(q)^2}
   \sums[h \le x/f(y_1)][q \mid h]
    \phi(q) 
     \int_{F(y_1)}^{F(y(h))} 
      (x - ht) \dd t.
\]
Since $P > \psL^B > h$ in this sum, by \eqref{eq:2.4}, we may 
rewrite the main term in the form 
\[
  \sum_{h \le x/f(y_1)} 
   \bigg\{ \sum_{q \mid h} \frac{\mu(q)^2}{\phi(q)} \bigg\}
    \int_{F(y_1)}^{F(y(h))} (x - ht) \dd t
   =
    \sum_{h \le x/f(y_1)} 
     \frac{h}{\phi(h)} 
      \int_{F(y_1)}^{F(y(h))} (x - ht) \dd t.
\]
Changing the limits of integration by an amount $O(1)$ incurs a 
further error 
\[
 \ll 
  x
   \sum_{h \le 2\psL^B} 
    \frac{h}{\phi(h)} 
     \ll 
      x^2 
       \psL^{-A}, 
\]
and this yields \eqref{eq:4.1}.

We simplify \eqref{eq:4.1} further using the formula 
\[
 \frac{\dd}{\dd t} 
  \big(t(x - {\textstyle \frac{1}{2}} ht)\big)
   =
    x - ht.
\]
For ease of comparison with \cite[Section 4]{BW11}, we write 
$W(h) = \frac{h}{\phi(h)}$.
The main term in \eqref{eq:4.1} is 
\[
 \sum_{h \le x/f(y)} 
  W(h)
   \bigg[
    t\bigg(x - \frac{1}{2}ht\bigg)
   \bigg]_{f(y_1)}^{f(y)}
    +
     \sum_{x/f(y) < h \le x/f(y_1)}
      W(h) 
       \bigg[
        t\bigg(x - \frac{1}{2}ht\bigg)
       \bigg]_{f(y_1)}^{x/h}.
\]
This can be rewritten as 
\begin{align*}
 & 
  \frac{f(y_1)^2}{2}
   \sum_{h \le x/f(y_1)} 
    \bigg\{ 
     \frac{W(h)}{h}
      \bigg(\frac{x}{f(y_1)}\bigg)^2 
       -
        2W(h) \frac{x}{f(y_1)} 
         + 
          W(h) h
    \bigg\}
  \\
   & 
    \hspace{60pt}
     -
      \frac{f(y)^2}{2} 
       \sum_{h \le x/f(y)} 
        \bigg\{ 
         \frac{W(h)}{h} 
          \bigg(\frac{x}{f(y)}\bigg)^2 
           - 
            2W(h) \frac{x}{f(y)}
             +
              W(h)h
        \bigg\}.
\end{align*}
Introducing the function 
\[
 \Theta(H)
  =
   \sum_{h \le H} 
    \frac{W(h)}{h}(H - h)^2,
\]
the main term in \eqref{eq:4.1} is 
\[
 \frac{1}{2}
  \bigg\{ 
   f(y_1)^2 \Theta\bigg(\frac{x}{f(y_1)}\bigg) 
    -
     f(y)^2 \Theta\bigg(\frac{x}{f(y)}\bigg)
  \bigg\}.
\]
Combining this with \eqref{eq:3.18}, we have 
\[
 2S_0 
  =
   f(y_1)^2 \Theta\bigg(\frac{x}{f(y_1)}\bigg) 
    -
     f(y)^2 \Theta\bigg(\frac{x}{f(y)}\bigg)
      +
       O\big(x^2\psL^{-A}\big).
\]

We have
\[
 \Theta(H) 
  = 
   \const 
    H^2\log H + 2\Gamma_0H^2 + H\log H + 2\Gamma_{-1}H + O(H^{1/2})
\]
from \cite[Section 5]{BW11} in the case 
$W(h) = \frac{h}{\phi(h)}$; here
\begin{equation}
 \label{eq:4.12}
  \const
   =
    \sum_{n = 1}^{\infty}
     \frac{\mu(n)^2}{n\phi(n)}.
\end{equation}
The constants $\Gamma_0, \Gamma_{-1}$ can be calculated 
explicitly; see \cite[p.\ 13]{BW11}.
Hence 
\begin{align}
 \label{eq:4.13}
  \begin{split}
  2S_0 
   & 
   =
    \const
     x^2 \log \bigg(\frac{f(y)}{f(y_1)}\bigg) 
     -
      x f(y)\log \bigg(\frac{x}{f(y)}\bigg)
       - 
        2\Gamma_{-1} x f(y)
  \\
   & 
    \hspace{150pt}
      + 
       O\big(x^{1/2} f(y)^{3/2} + x^2\psL^{-A}\big).
  \end{split}
\end{align}  
 

\section{Completion of the proof of Theorems \ref{thm:1} and \ref{thm:2}}
 \label{sec:5} 
 
We begin by using the identity 
\[
 \frac{q}{\phi(q)}
  =
   \sum_{r \mid q} 
    \frac{\mu(r)^2}{\phi(r)}
\]
to evaluate $\Phi_F(y_1,y)$ asymptotically.
For $y_1 < z \le y$, 
\begin{align}
 \label{eq:5.1}
  \begin{split}
  \sum_{y_1 < k \le z}
   \frac{f(k)}{\phi(f(k))} 
    & 
    =
     \sum_{r \le f(z)} \frac{\mu(r)^2}{\phi(r)}
      \sums[y_1 < k \le z][f(k) \equiv 0 \bmod r] 1
    \\
    & 
    =
     \sum_{r \le f(z)} 
      \frac{\mu(r)^2}{\phi(r)} 
       \bigg(
        \frac{z - y_1}{r}
         + 
          O\big(z \psL^{-4A}\big)
       \bigg)
    \\
    & 
     =
      \const 
       (z - y_1) 
        + 
         O\big(z \psL^{-3A}\big),
  \end{split}
\end{align}
where we have used Lemma \ref{lem:5} for the second last equality, 
and where $\const$ is the constant in \eqref{eq:4.12}.
Replacing $\frac{f(k)}{\phi(f(k))}$ by $\frac{F(k)}{\phi(f(k))}$ 
in \eqref{eq:5.1} introduces an error that is
\[
 \ll 
  \sum_{y_1 < k \le z} 
   \frac{1}{\phi(f(k))} 
    \ll
     \psL 
      \sum_{y_1 < k \le z}
       \frac{1}{f(k)} 
        \ll
         \psL 
          \sum_{y_1 < k \le z} 
           \frac{\psL^B}{x}
            \ll 
             \frac{\psL^{B + 1} z}{x}.
\]
Hence
\begin{equation}
 \label{eq:5.2}
  N(z)
   \defeq 
    \sum_{y_1 < k \le z}
     \frac{F(k)}{\phi(f(k))}
      -
       \const 
        (z - y_1) 
         \ll 
          z\psL^{-3A}.
\end{equation}

Now
\begin{align*}
 \Phi_F(y_1,y) 
  -
   \const
    \int_{y_1}^y \frac{F'(t)}{F(t)} \dd t 
 & 
   =
      \int_{y_1}^y \frac{F'(t)}{F(t)} \dd N(t) 
 \\
 & 
     =
        \bigg[\frac{F'(t)}{F(t)}N(t)\bigg]_{y_1}^y
         -
          \int_{y_1}^y N(t) 
           \frac{\dd}{\dd t} 
            \bigg(\frac{F'(t)}{F(t)}\bigg) 
             \dd t.          
\end{align*}
Since $F'/F$ is monotonic, we deduce from \eqref{eq:5.2} that, for 
some $w \in [y_1,y]$,
\[
 \Phi_F(x,y) 
  =
   \const
    \int_{y_1}^y \frac{F'(t)}{F(t)}
     +
      O\bigg(\frac{F'(w)}{F(w)}w \psL^{-3A}\bigg)
       =
        \const
         \log\bigg(\frac{F(y)}{F(y_1)}\bigg)
          +
           O\big(\psL^{-A}\big)
\]
(recalling \ref{eq:3.2}).
Noting that 
\begin{align*}
 \log \frac{F(y)}{F(y_1)}
  -
   \log \frac{f(y)}{f(y_1)}
 & 
   =
     \log\bigg(1 + \frac{F(y) - f(y)}{f(y)}\bigg)
      -
       \log\bigg(1 + \frac{F(y_1) - f(y_1)}{f(y_1)}\bigg)
 \\
 & 
  =
   O\bigg(\frac{1}{f(y_1)}\bigg)
    =
     O\big(\psL^{-A}\big)
\end{align*}
we have the more convenient expression 
\begin{equation}
 \label{eq:5.3}
 \Phi_F(y_1,y)
  =
   \const 
    \log\bigg(\frac{f(y)}{f(y_1)}\bigg)
     +
      O\big(\psL^{-A}\big).
\end{equation}

We now substitute \eqref{eq:4.13} and \eqref{eq:5.3} into the 
expression for $V_F'(x,y)$ obtained in \eqref{eq:3.11}.
This gives 
\begin{align*}
 V_F'(x,y)
  & 
   =
    x
     f(y)
      \log f(y) 
       +
        f(y)
         \bigg( 
          \sum_{p \le x} (\log p)^2 
           - 
            x
             \log x
         \bigg)
          -
           2
            \Gamma_{-1} 
             x
              f(y) 
 \\
  & 
   \hspace{180pt}
    +    
     O
      \big(
       x^{1/2}
        f(y)^{3/2}
         +
          x^2\psL^{-A}        
      \big).
\end{align*}
By the prime number theorem and partial summation, 
\[
 V_F'(x,y)
  =
   x
    f(y)
     \log f(y) 
      +
       C
        x
         f(y)
          +
           O
            \big(
             x^{1/2} 
              f(y)^{3/2}
               +
                x^2 
                 \psL^{-A}
            \big)
\]
with $C = -(2\Gamma_{-1} + 1)$.

The additional sum required to complete the proof of Theorems 
\ref{thm:1} and \ref{thm:2} is 
\[
 \sum_{Y < k \le y_1} 
  F'(k)
   \sums[\ell = 1][(\ell,f(k)) = 1]^{f(k)} 
    E(x; f(k), \ell)^2,
\]
where $Y = y_0(F)$ (Theorem \ref{thm:1}), 
$Y = \exp\big((\log \sL)^{C_1}\big)$ (Theorem \ref{thm:2}).
By a splitting-up argument it suffices to show that when 
$Y \le Q < F(y_1)$, we have
\begin{equation}
 \label{eq:5.4}
 \sum_{Q < [F(k)] \le 2Q} 
  F'(k) 
   \sums[\ell = 1][(\ell,f(k)) = 1]^{f(k)}
    E(x; f(k), \ell)^2 
     \ll
      x^2\psL^{-2A}.
\end{equation}
This is a straightforward consequence of Lemma \ref{lem:1} in the 
case of Theorem \ref{thm:1}.
In the case of Theorem \ref{thm:2}, let 
\[
 F(K) = Q, 
  \quad  
   F(K_1) = 2Q
\]
so that 
\[
 F'(k) 
  \asymp 
   (\log K)^{\gamma - 1} 
    K^{-1} 
     Q
      \quad 
       (k 
        \in 
         [K,K_1]).
\]
Now the mean value theorem yields 
\[
 \frac{|S_Q|}{Q} 
  \asymp 
   \frac{K_1 - K}{F(K_1) - F(K)}
    \asymp 
     \frac{K}{Q(\log K)^{\gamma - 1}}.
\]
The left-hand side of \eqref{eq:5.4} is 
\[
 \ll 
  \frac{(\log K)^{\gamma - 1}Q}{K}
   \sum_{q \, \in \, S_Q} 
     \sums[\ell = 1][(\ell,q) = 1]^q
     E(x; q, \ell)^2
      \ll 
       \frac{Q}{|S_Q|} 
        \sum_{q \, \in \, S_Q}
         \sums[\ell = 1][(\ell,q) = 1]^q
          E(x; q, \ell)^2
           \ll
            x^2 
             \psL^{-2A}  
\]
by Lemma \ref{lem:2}. 
This completes the proof of Theorems \ref{thm:1} and \ref{thm:2}.

The constant $C$ may be evaluated using material from 
\cite[Section 5]{BW11}, with the function $\rho(p)$ replaced by 
$1$ to yield the desired function $W(h) = h/\phi(h)$.
Let 
\[
 D(s) 
  =
   \zeta(s + 1)
    \zeta(s + 2)
     E_2(s), 
\]
where $\zeta$ is Riemann's zeta function and 
\begin{equation}
 \label{eq:5.5}
  E_2(s)
   =
    \prod_p
     \bigg( 
      1 + \frac{p^{-s}}{p^2(p - 1)} - \frac{p^{-2s}}{p^3(p - 1)}
     \bigg)
      \quad 
       (\Re(s) > -3/2).
\end{equation}
Then in the notation used just before \eqref{eq:4.12}, we have the 
residue formula 
\begin{equation}
 \label{eq:5.6}
  \text{Res}
   \bigg(
    \frac{D(s)H^{s + 2}}{s(s + 1)(s + 2)}, -1
   \bigg)
    =
     -
      \zeta(0) 
       H 
        \log H
         + 
          \Gamma_{-1}
           H.
\end{equation}
We use $E_2(-1) = 1$ (see \cite[p.\ 301]{BW11}).
We also need 
\begin{equation}
 \label{eq:5.7}
  E_2'(-1)
   =
    \sum_{p}
     \frac{\log p}{p(p - 1)},
\end{equation}
which can be obtained from \eqref{eq:5.5} by logarithmic 
differentiation.
Now we have the Laurent expansions near $-1$:
\begin{align*}
 G(s)
  & 
   \defeq 
    \frac{\zeta(s + 1)H^{s + 2}E_2(s)}{s(s + 2)}
     =
      -\zeta(0)
        H
         + 
          G'(-1)
           (s + 1)
            + 
             \cdots , 
 \\
  \frac{\zeta(s + 2)}{s + 1}
   & 
    =
     \frac{1}{(s + 1)^2} 
      (1 + \gamma_0(s + 1) + \cdots ).
\end{align*}
See Ivi\'c \cite[p.\ 4]{I03} for the coefficient $\gamma_0$ 
(Euler's constant) in the latter expansion.
We are led immediately to 
\begin{equation}
 \label{eq:5.8}
  \text{Res}
   \bigg(
    \frac{D(s)H^{s + 2}}{s(s + 1)(s + 2)}, -1
   \bigg)
    =
     -\gamma_0 
       \zeta(0) 
        H
         +
          G'(-1).
\end{equation}
A short calculation yields 
\begin{equation}
 \label{eq:5.9}
  G'(-1)
   =
    -
     \big\{
      \zeta'(0)
       + 
        \zeta(0) 
         E_2'(-1)
     \big\}
      H
       -
        \zeta(0)
         H
          \log H.
\end{equation}
Combining \eqref{eq:5.6}--\eqref{eq:5.9}, 
\[
 \Gamma_{-1}
  =
   \zeta(0) 
    \bigg(
     -\gamma_0 
      -
       \sum_{p}
        \frac{\log p}{p(p - 1)}
    \bigg)
     -
      \zeta'(0)
\]
and 
\begin{equation}
 \label{eq:5.10}
 C
  =
   -(2\Gamma_{-1} + 1)
    =
     2
      \zeta(0) 
    \bigg(
     \gamma_0 
      +
       \sum_{p}
        \frac{\log p}{p(p - 1)}
    \bigg)
     +
      2
       \zeta'(0)
        - 
         1.
\end{equation}


\end{document}